\documentclass{article}
\usepackage{amsfonts}
\usepackage{amsmath}
\usepackage{theorem}
\usepackage{array}
\usepackage{graphics}
\usepackage[usenames]{color}
\usepackage[a4paper]{geometry}
\usepackage{amssymb}
\usepackage[all]{xypic}
\usepackage{graphicx}
\newtheorem{theorem}{Theorem}

\newtheorem{corollary}[theorem]{Corollary}

\newtheorem{definition}[theorem]{Definition}

{ \theorembodyfont{\rmfamily}

\newtheorem{remark}[theorem]{Remark}
}

\newtheorem{proposition}[theorem]{Proposition}

\newenvironment{proof}[1][Proof]{\textbf{#1.} }{\ \rule{0.5em}{0.5em}}

\newcommand{\F}{\mathbb{F}_q}

\begin{document}

\title{On the Deuring Polynomial for Drinfeld Modules in Legendre Form}
\author{Alp Bassa\footnote{Alp Bassa is supported by the BAGEP Award of the Science Academy with funding supplied by Mehve{\c s} Demiren in memory of Selim Demiren.} \quad and\quad  Peter Beelen\footnote{Peter Beelen is supported by The Danish Council for Independent Research (Grant No. DFF--4002-00367).}}
\date{}
\maketitle
\abstract{We study a family $\psi^{\lambda}$ of $\mathbb F_q[T]$-Drinfeld modules, which is a natural analog of Legendre elliptic curves. We then find a surprising recurrence giving the corresponding Deuring polynomial $H_{p(T)}(\lambda)$ characterising supersingular Legendre Drinfeld modules $\psi^{\lambda}$ in characteristic $p(T)$.}

\section{Introduction}

In this paper we motivate the definition of a family of Drinfeld modules introduced in 2011 in a preliminary version of this work \cite{unred} and named Legendre Drinfeld modules by El-Guindy in \cite{elguindy}. We investigate an analog of the Deuring polynomial for these Legendre Drinfeld modules. We show how these polynomials can be obtained very easily using a simple recurrence.

More precisely, let $A=\mathbb F_q[T]$ and let $L$ be an $A$-field with the structure homomorphism $\gamma:A\to L$. We assume that $L$ has $A$-characteristic ${\rm ker\ }\gamma=\langle p(T) \rangle$, with $p(T)\neq T$ and $\deg p(T)=d$.

An $\mathbb F_q[T]$-Drinfeld module $\psi^{\lambda}$ in Legendre form is given by
\begin{equation*}  \psi^{\lambda}_T=\frac{\gamma(T)}{(\lambda^q-\lambda)^{q-1}}\tau^2-\Bigl(\frac{\gamma(T)}{(\lambda^q-\lambda)^{q-1}}+\gamma(T)\Bigr)\tau+\gamma(T)\end{equation*}
with $\lambda \notin \mathbb F_q$. It can been seen that every Drinfeld module is isomorphic to a Drinfeld module in Legendre form.
The Deuring polynomial $H_{p(T)}$ for the prime $p(T)$ is defined as
$$H_{p(T)}(s)=\prod(s-\lambda),$$
where the product is taken over all $\lambda \in \overline{L}$ such that $\psi^{\lambda}$ is supersingular. $H_{p(T)}(s)$ characterizes the supersingular $\lambda$-invariants.

We show that $H_{p(T)}(s)$ can be obtained easily using a recurrence:
\begin{theorem}
For $i\geq -1$ we define the polynomials $U_i(s)\in A[s]$  recursively by:
$$U_{-1}(s)=0  \makebox{, } U_0(s)=1$$
and
\begin{equation*}
U_{i+1}(s)=\bigl((s^q-s)^{q-1}+\frac{1}{T^{q-1}}\bigr)^{q^i}U_i(s)-\frac{T^{q^i}-T}{T^{q^{i+1}}}(s^q-s)^{(q-1)q^{i-1}}U_{i-1}(s),\quad \text{for } i\ge 0.
\end{equation*}
Let $p(T)\neq T$ be an irreducible polynomial in $A$ with $\deg p(T)=d$. Then
$$H_{p(T)}(s)=U_d^{(p(T))}(s),$$
where $U_d^{(p(T))}(s)$ denotes the reduction of $U_d(s)$ modulo $p(T)$.
\end{theorem}

Note that $H_{p(T)}(s)$ is obtained by reduction of $U_d(s)$, which depends on the degree of $p(T)$ only! Polynomials with such properties are called ``universal supersingular polynomials".

Universal supersingular invariants at the level of the $j$-line have previously been investigated by Cornelissen \cite{cor1,cor2}, El-Guindy--Papanikolas \cite{elgpapa}, El-Guindy \cite{elguindy}, Gekeler \cite{ge4, ge5} and Schweizer \cite{schweizer2}.

Alternatively, one can consider Drinfeld modules at a higher level, which we previously have called the $u$-line in \cite{unred}, and which are weak analogs of Legendre elliptic curves.
For $\Delta\neq 0$ we consider the Drinfeld module
\begin{equation*} \psi^{\Delta}_T=\Delta\tau^2-(\Delta+\gamma(T))\tau+\gamma(T)
\end{equation*}
As shown below, the $\Delta$-line corresponds to the modular curve $X_0(T)$ and forms a degree $q+1$ cover of the $j$-line $X(1)$. An analogous Deuring polynomial $h_{p(T)}(s)$ characterising supersingular $\Delta$-invariants in $A$-characteristic $p(T)$ can be defined as well and we show below that it is obtained similarly as the reduction of recursively given polynomials $u_i(s)$ modulo $p(T)$, see Theorem~\ref{thm:main}.

Classically, for an elliptic curve $E_\lambda/\mathbb C$ in Legendre form the integrals giving the periods of the corresponding lattice can be expressed in terms of  the hypergeometric function $_2F_1(1/2,1/2,1;s)$. The Deuring polynomial, characterising $\lambda$-invariants of supersingular Legendre elliptic curves in characteristic $p$ is obtained as the truncation of the expression for the real period in terms of $_2F_1(1/2,1/2,1;s)$ (by reducing it modulo $p$). El-Guindy \cite{elguindy} has obtained formulas for periods of Legendre Drinfeld modules and has shown, using the characterisation of the Deuring polynomial $h_{p(T)}(s)$ given in a preliminary version \cite{unred} of this paper, that their truncations precisely gives the Deuring polynomials $h_{p(T)}(s)$.

Note that in \cite{ge3} Gekeler gives and studies an analog of the Deuring polynomial for the family of Drinfeld modules given by $\phi_T=\tau^2+\lambda\tau+T$. As for the $\Delta$-family above, the $\lambda$-line in \cite{ge3} is a degree $q+1$ covering of the $j$-line, given however by the relation $j=\lambda^{q+1}$. The corresponding Deuring polynomial also behaves similarly as its classical counterpart. However the analogy with the classical case is not as strong.

The results in this work have appeared in a preliminary arXiv version \cite{unred} in 2011. Since the aim in \cite{unred} was the construction of asymptotically optimal towers of function fields, results on the Deuring polynomial were hidden in the exposition and their proof depended on the optimality of the towers. We give an exposition independent of towers and a self-contained proof using a result about correspondences from \cite{beelen}. We would like to thank Ahmad El-Guindy for pointing out that these results are of independent interest.

\section{Drinfeld modules in Legendre form}

To motivate the definition of Legendre Drinfeld modules, we start be recalling the situation in the case of elliptic curves. For details see \cite{sil}. Let $k$ be a field with ${\rm char}\ k \ne 2$. Let $E$ be an elliptic curve over $\overline{k}$, given by the equation  $y^2=(x-e_1)(x-e_2)(x-e_3)$, with the $e_i$ distinct. The $2$-torsion points are given by $(e_1,0),(e_2,0), (e_3,0)$ and $\infty$. The admissible change of variables
$$x=(e_2-e_1)x'+e_1,\quad y=(e_2-e_1)^{3/2}y'$$ maps the ordered torsion points respectively on $(0,0)$, $(1,0)$, $(\lambda,0)$ and $\infty$ and gives the isomorphic elliptic curve $E'$ in Legendre form
\begin{equation}\label{legendre} y^2=x(x-1)(x-\lambda),\text{ with } \lambda=\frac{e_3-e_1}{e_2-e_1}.\end{equation}
Hence for every elliptic curve, we can find an isomorphic elliptic curve in Legendre form, so that the isomorphism maps two chosen nontrivial $2$-torsion points $(e_1,0)$ and $(e_2,0)$ on the $2$-torsion points $(0,0)$ and $(1,0)$, respectively.
For an elliptic curve $E$ the ordering on the  $2$-torsion points $(e_1,0),(e_2,0), (e_3,0)$ corresponds to a choice of ordered  basis $(e_1,0),(e_2,0)$ for the $\mathbb Z/2\mathbb Z$-vector space $E[2]$, and the basis vectors are mapped on $(0,0)$ and $(1,0)$ respectively by the change of variable. Hence we can see $\lambda$ as a uniformizer for the modular curve $X(2)$. The $j$-invariant of the Legendre Elliptic Curve (\ref{legendre}) is given by
$$j(\lambda)=2^8\frac{(\lambda^2-\lambda+1)^3}{\lambda^2(\lambda-1)^2}.$$
This gives the relation between $\lambda$ and the function $j$ on $X(1)$.
We obtain a surjective map $\overline{k}\backslash\{0,1\}\to \overline{k}\quad \lambda\mapsto j(\lambda)$, which is six-to-one, except above $j=0$ and $j=1728$.

Suppose $k$ is a field of positive characteristic $p\neq 2$. The elliptic curve $E$ given by Equation~\eqref{legendre} is supersingular, if and only if the coefficient of $x^{p-1}$ in $\bigl(x(x-1)(x-\lambda)\bigr)^{(p-1)/2}$ is zero. Let $m=(p-1)/2$ and define the Deuring polynomial
$$H_p(s)=\sum_{i=0}^m \binom{m}{i}^2s^i.$$ The coefficient of $x^{p-1}$ in $\bigl(x(x-1)(x-\lambda)\bigr)^{(p-1)/2}$ is given up to a sign by $H_p(\lambda)$. Therefore for $\lambda\neq 0,1$, the elliptic curve $y^2=x(x-1)(x-\lambda)$ is supersingular if and only if $H_p(\lambda)=0$.

The analogous picture for Drinfeld modules looks as follows: Let $A=\mathbb F_q[T]$ and let $L$ be an $A$-field with the structure homomorphism $\gamma:A\to L$. The role of the prime $2$ above will be played by the prime $T$, so we assume throughout that $L$ has $A$-characteristic ${\rm ker\ }\gamma=\langle p(T) \rangle\neq \langle T\rangle$, with $\deg p(T)=d$. Let $\varphi$ be a Drinfeld $A$-module defined over $\overline{L}$ of rank $2$. Let $e_1,e_2$ be an $\mathbb F_q[T]/\langle T \rangle$ basis for the $T$-torsion $\varphi[T]$. As above we specify an isomorphism using the torsion point $e_1$:  consider the isomorphic Drinfeld module $\psi$, with an isomorphism from $\psi$ to $\varphi$ given by the degree zero additive polynomial $e_1=e_1 \tau^0$:
$$e_1 \cdot \psi_T=\varphi_T\cdot e_1.$$
The $T$-torsion point $e_1$ of $\varphi$ is mapped onto $T$-torsion point $1$ of $\psi$ and if
$$\psi_T=\Delta\tau^2+g\tau+\gamma(T),\quad \Delta\neq 0$$
then $\psi_T(1)=\Delta+g+\gamma(T)=0$, so $g=-(\Delta+\gamma(T))$ and
\begin{equation}\label{uline} \psi_T=\Delta\tau^2-(\Delta+\gamma(T))\tau+\gamma(T)=(\Delta \tau-\gamma(T))(\tau-1)
\end{equation}
Note that the $\mathbb F_q[T]$-submodule $\mathbb F_q\cdot e_1$ of $\varphi[T]$ of rank 1 is mapped onto the submodule $\mathbb F_q$ of $\psi[T]$. Hence $\Delta$ can be seen as a uniformizer for the modular curve $X_0(T)$.

The $T$-torsion point $e_2$ is mapped onto $\lambda\notin \mathbb F_q\cdot 1=\mathbb F_q$. So
$$\psi_T(\lambda)=\Delta \cdot \lambda^{q^2}-(\Delta+\gamma(T))\lambda^q+\gamma(T)\lambda=0.$$
Since $\lambda^q\neq \lambda$, we obtain
\begin{equation} \label{deltalambda} \Delta=\frac{\gamma(T)}{(\lambda^q-\lambda)^{q-1}}.\end{equation}
Hence we see that every Drinfeld module $\varphi$ together with an ordered basis $e_1,e_2$ for the $T$-torsion is isomorphic to a Drinfeld module $\psi$ of the form
\begin{equation} \label{lline} \psi_T=\frac{\gamma(T)}{(\lambda^q-\lambda)^{q-1}}\tau^2-\Bigl(\frac{\gamma(T)}{(\lambda^q-\lambda)^{q-1}}+\gamma(T)\Bigr)\tau+\gamma(T)\end{equation}
The basis $e_1, e_2$ is mapped under this isomorphism onto $1,\lambda$. Hence we can see $\lambda$ as a uniformizer for the modular curve $X(T)$. The $j$-invariant of the Drinfeld module (\ref{lline}) is given by
$$j(\lambda)=\frac{\bigl(\Delta+\gamma(T)\bigr)^{q+1}}{\Delta}=\frac{\gamma(T)^q\bigl(1+(\lambda^q-\lambda)^{q-1}\bigr)^{q+1}}{(\lambda^q-\lambda)^{q^2-q}}.$$
This gives the relation between $\lambda$ and the function $j$ on $X(1)$.
We obtain a surjective map $\overline{L}\backslash\mathbb F_q\to \overline{L}\quad \lambda\mapsto j(\lambda)$, which is $(q^3-q)$-to-one, except above $j=0$.

Next we want to define the analog of the Deuring polynomial. We are looking for a polynomial $H_{p(T)}(s)$ such that the Drinfeld module $\psi^{\lambda}$ with
$$ \psi^{\lambda}_T=\frac{\gamma(T)}{(\lambda^q-\lambda)^{q-1}}\tau^2-\Bigl(\frac{\gamma(T)}{(\lambda^q-\lambda)^{q-1}}+\gamma(T)\Bigr)\tau+\gamma(T)=\Bigl(\frac{\gamma(T)}{(\lambda^q-\lambda)^{q-1}} \tau-\gamma(T)\Bigr)(\tau-1)$$
is supersingular if and only if $H_{p(T)}(\lambda)=0$. Since $L$ has $A$-characteristic $p(T)$, the kernel of the multiplication by $p(T)$ map, $\psi^{\lambda}[p(T)]$, is a free $\mathbb F_q[T]/\langle p(T)\rangle$-module of rank $0$ or $1$. It has rank $0$ (and $\psi^{\lambda}$ is supersingular) exactly when the coefficient of $\tau^d$ in $\psi^{\lambda}_{p(T)}$ is zero, where $d=\deg p(T)$.

Since $\Delta$ and $\lambda$ are related by a simple relation (\ref{deltalambda}), we start by determining supersingular Drinfeld modules of the form (\ref{uline}).
For $\psi_T=\Delta\tau^2-(\Delta+\gamma(T))\tau+\gamma(T)=(\Delta \tau-\gamma(T))(\tau-1)$ let
\begin{equation}\label{eq:psipt}
\psi_{p(T)}=\sum_{k=0}^{2d}g_k(\Delta)\tau^k.
\end{equation}
A direct bookkeeping argument for the degree of $\Delta$ in the coefficient of $\tau^d$ in \eqref{eq:psipt} shows that $g_{d}(\Delta)$ has leading term $(-1)^d \Delta^{(q^d-1)/(q-1)}.$
Therefore $h_{p(T)}(\Delta):=(-1)^d g_d(\Delta)$ is a monic polynomial of degree $(q^d-1)/(q-1)$ in $\Delta.$
Furthermore, we have $g_k(\Delta)=0$ for $0\leq k<d$ and
\begin{equation}\label{eq:g2d}
g_{2d}=\Delta^{1+q^2+q^4+\cdots+q^{2d-2}}.
\end{equation}
By an immediate adaptation of calculations  in \cite{ge3}, we obtain the following results for $h_{p(T)}$: Considering the coefficient of $\tau^k$ in $\psi_{p(T)}\psi_T=\psi_T\psi_{p(T)}$ we obtain the recurrence
\begin{align}\label{recurrence}
\gamma(T^{q^k}-T)\cdot g_k&=g_{k-1}\Delta^{q^{k-1}}-g_{k-2}\Delta^{q^{k-2}}+\bigl(g_{k-2}^{q^2}-g_{k-1}^{q}\bigr)\Delta-\gamma(T)g_{k-1}^{q}+g_{k-1}\gamma(T^{q^{k-1}}) \notag \\
&=g_{k-1}\Omega^{q^{k-1}}-g_{k-1}^q\Omega-\bigl(g_{k-2}\Delta^{q^{k-2}}-g_{k-2}^{q^2}\Delta\bigr)
\end{align}
with $\Omega=\Delta+\gamma(T)$.
It follows that $h_{p(T)}(\Delta)|g_k(\Delta)$ for $d\leq k<2d$. Hence roots of $h_{p(T)}(\Delta)$ are roots of $g_k(\Delta)$ for $0 \le k < 2d$. Therefore, the Drinfeld module in \eqref{uline} is supersingular if and only if $\Delta$ is a root of $h_{p(T)}(\Delta)$. As in \cite{ge3} one can show that $h_{p(T)}(\Delta)$ is separable, by considering equation \eqref{recurrence} for $k=2d+1$.

By \eqref{deltalambda}, the monic polynomial
\begin{equation} \label{deuringrelation}
H_{p(T)}(\lambda):=\Bigl(\dfrac{(\lambda^q-\lambda)^{q-1}}{\gamma(T^q)}\Bigr)^{\frac{q^d-1}{q-1}}h_{p(T)}\Bigl(\frac{\gamma(T)}{(\lambda^q-\lambda)^{q-1}}\Bigr)
\end{equation}
will be the analog of the Deuring Polynomial in $A$-characteristic $p(T)$. It is easily seen to be separable and its roots correspond to supersingular Drinfeld modules in Legendre form (\ref{lline}) in $A$-characteristic $p(T)$.

\section{Equations for the Drinfeld modular curves $X_0(T^n)$}
In this section we find explicit equations for the curves $X_0(T^n)$ over $\mathbb F_q[T,1/T]$ for any $n\geq 1$. Equations for the reductions of these curves modulo $T-1$ (i.e., $\gamma(T)=1$) have been obtained by Elkies in \cite{elkiesd}. The equations derived in this section hold for any $A$-characteristic with $\gamma(T)\neq 0$.

Denote by $\mathbb F$ the field $\mathbb F_q(\gamma(T))$. Let $F_n$ be the function field of $X_0(T^n)$ over $\mathbb F$. We know that $j_1=j$ is a uniformizer of $X(1)$ and hence $F_0=\mathbb F(j_1)$. Moreover $\Delta_0=\Delta$ is a uniformizer for $X_0(T)$: given a Drinfeld module $\varphi$ together with $T$-isogeny $\iota$, we can find an isomorphic Drinfeld module $\psi$ of the form (\ref{uline}), such that the isomorphism maps the kernel of $\iota$ onto $\mathbb F_q$. So the pair $(\varphi,\iota)$ will correspond to the pair $(\psi,\kappa)$, with $\kappa=\tau-1$. Hence $F_1=\mathbb F(\Delta_0)$, with
$$j_1=\frac{\bigl(\Delta_0+\gamma(T)\bigr)^{q+1}}{\Delta_0}.$$
We have the isogeny $\kappa:\psi\to\psi'$. The isogenous Drinfeld module $\psi'$ will be given by
$$\psi'_T=(\tau-1)(\Delta_0\tau-\gamma(T))=\Delta_0^q\tau^2-(\Delta_0+\gamma(T^2))\tau+\gamma(T)$$ and will have $j$ invariant
\begin{equation}\label{eq1}
j_2=\frac{\bigl(\Delta_0+\gamma(T^q)\bigr)^{q+1}}{\Delta_0^q}.
\end{equation}
We have $F_1=\mathbb F(\Delta_0)=\mathbb F(j_1,j_2)$. This description for $X_0(T)$ can also be found in \cite{schweizer}.
By iteration, equations for $F_n$ can be found. Let $\iota'$ be a $T$-isogeny from $\psi'$. The pair $(\psi',\iota')$ will correspond to $\Delta_1$ with
\begin{equation} \label{eq2} j_2=\frac{\bigl(\Delta_1+\gamma(T)\bigr)^{q+1}}{\Delta_1}.\end{equation}
From Equation~(\ref{eq1}) and (\ref{eq2}) we obtain
\begin{equation}\label{eq:schweizer}
\dfrac{(\Delta_0+\gamma(T^q))^{q+1}}{\Delta_0^q}=\dfrac{(\Delta_1+\gamma(T))^{q+1}}{\Delta_1}.
\end{equation}
We have the factorization
\begin{multline*}\dfrac{(\Delta_0+\gamma(T^q))^{q+1}}{\Delta_0^q}-\dfrac{(\Delta_1+\gamma(T))^{q+1}}{\Delta_1}\\=\left(\Delta_0-\frac{\gamma(T^{q+1})}{\Delta_1}\right)\left(1+\frac{\gamma(T^{q^2})}{\Delta_0^q}
-\left(\Delta_1-\frac{\gamma(T^{q+1})}{\Delta_0}\right)^{q-1}\left(\frac{\Delta_1}{\Delta_0}+\frac{\gamma(T)}{\Delta_0}\right) \right).
\end{multline*}

The linear factor corresponds to the case where $\iota'$ is the dual isogeny of $\kappa$. If this is not the case, their composition will give a $T^2$-isogeny of $\psi$.

Hence the function field of $X_0(T^2)$ is given by $F_2=\mathbb F(\Delta_0,\Delta_1)$, with
$$1+\frac{\gamma(T^{q^2})}{\Delta_0^q}-\left(\Delta_1-\frac{\gamma(T^{q+1})}{\Delta_0}\right)^{q-1}\left(\frac{\Delta_1}{\Delta_0}+\frac{\gamma(T)}{\Delta_0}\right)=0.$$

Since the genus of $X_0(T^2)$ is zero, our first task is to find a generator of its function field. The element $\Theta_0=(\Delta_0\Delta_1-\gamma(T^{q+1}))/(\Delta_0+\gamma(T^q))$ is a generator of the function field $F_2$, and we find
$$\Delta_0=\Theta_0^{q-1}(\Theta_0+\gamma(T)) \makebox{ and } \Delta_1=\frac{(\Theta_0+\gamma(T))^q}{\Theta_0^{q-1}}.$$

It will be convenient to use a slightly different generator, namely $Y_0=-(\Theta_0+\gamma(T))/\gamma(T)$. In terms of $Y_0$, we find
\begin{equation}\label{eq:correspondence}
\Delta_0=-\gamma(T^q)(Y_0+1)^{q-1}Y_0 \makebox{ and } \Delta_1=-\gamma(T)\frac{Y_0^q}{(Y_0+1)^{q-1}}.
\end{equation}
The function field $F_{n+2}$, $n\ge 0$ can be generated by elements $Y_0,\dots,Y_{n}$ satisfying the equations $$-\gamma(T^q)(Y_{i}+1)^{q-1}Y_i=-\gamma(T)\frac{Y_{i-1}^q}{(Y_{i-1}+1)^{q-1}} \quad (1\le i \le n)$$ which simplifies to the equations
\begin{equation}\label{eq:unreducedElkies}
(Y_{i}+1)^{q-1}Y_i=\frac{Y_{i-1}^q}{\gamma(T^{q-1})(Y_{i-1}+1)^{q-1}} \quad (1\le i \le n).
\end{equation}
In fact we have shown the following:
\begin{theorem}
Let $\mathbb{F}=\mathbb{F}_q(\gamma(T))$. Denote for $n\ge 0$ by $F_n$ the function field of the Drinfeld modular curve $X_0(T^n)$ over $\mathbb F$. Then the tower of function fields $\mathcal F = (F_n)_{n\ge 0}$ can be given recursively as follows:
\begin{itemize}
\item $F_0=\mathbb F(j)$,
\item $F_1=\mathbb F(\Delta)$, with
$$j=\frac{\bigl(\Delta+\gamma(T)\bigr)^{q+1}}{\Delta}.$$
\item $F_2=\mathbb F(Y_0)$ with
$$\Delta=-\gamma(T^q)(Y_0+1)^{q-1}Y_0$$
and $$F_{n}=F_{n-1}(Y_{n}),\quad \text{with}\quad
(Y_{n}+1)^{q-1}Y_n=\frac{Y_{n-1}^q}{\gamma(T^{q-1})(Y_{n-1}+1)^{q-1}}$$
for $n\geq 3$.
\end{itemize}
\end{theorem}

\section{A Recurrence for the Deuring Polynomial for Drinfeld modules in Legendre Form}\label{section:three}

From now on, we assume that the $A$-characteristic is finite, i.e. $\ker \gamma = \langle p(T) \rangle$, with $\deg p(T)=d$. Let $L$ be an algebraic closure of $\mathbb F=A/\langle p(T) \rangle$. Equation \eqref{eq:correspondence} gives rise to the correspondence $L(\Delta_0)\subseteq L(Y_0)\supseteq L(\Delta_1):$
\begin{center}
\begin{minipage}{0.3\linewidth}
\xy
\xymatrix@dr{
L(X_0(T^2))= L(Y_0) \ar@{-}[rr] \ar@{-}[dd] & &L(X_0(T))=L(\Delta_1) \\
\\
L(X_0(T))=L(\Delta_0) \\
}
 \endxy
 \end{minipage}\hspace{-0.5cm}
\begin{minipage}{0.4\linewidth}
\begin{eqnarray} \label{correspondence}
 &\Delta_0&=-\gamma(T^q)(Y_0+1)^{q-1}Y_0\notag \\
 \notag \\
&&\text{and}\\
\notag \\
&\Delta_1&=-\gamma(T)\frac{Y_0^q}{(Y_0+1)^{q-1}} \notag .
\end{eqnarray}
 \end{minipage}
\end{center}

Define the field isomorphism $\iota: L(\Delta_0) \to L(\Delta_1)$ by $\iota(\Delta_1)=\Delta_0$. This map together with the correspondence $L(\Delta_0)\subseteq L(Y_0)\supseteq L(\Delta_1),$ gives rise to a directed graph $G=(V,E).$ The vertex-set $V$ is given by the set of places of $L(\Delta_0)$, while $(P,Q) \in E$ if and only if there exists a place $R$ of $L(Y_0)$ such that $R\cap L(\Delta_0)=P$ and $R\cap L(\Delta_1)=\iota(Q).$ See \cite{beelen} for details. Note that a vertex $P$ corresponds to an isomorphism class of Drinfeld modules together with a $T$-isogeny (the level structure) and that an edge $(P,Q)$ corresponds to a $T$-isogeny between the corresponding Drinfeld modules preserving the level structure.

Note that since the $A$-characteristic is finite, we can consider the set of places $V_{s} \subset V$ of $L(\Delta_0)$ corresponding to the supersingular points on $X_0(T)$. From \cite[Thm.5.9]{ge3} one obtains immediately that $|V_s|=(q^d-1)/(q-1)$.

Note that
\begin{equation}
h_{p(T)}(s)=\prod_{P \in V_s} (s-\Delta_0(P)).
\end{equation} If two Drinfeld modules $\phi$ and $\psi$ are isogenous, then $\phi$ is supersingular if and only if $\psi$ is supersingular. This means that if $(P,Q)$ is an edge  of $G$, then $P \in V_s$ if and only if $Q \in V_s$. Moreover, since in the extension $L(Y_0)/L(\Delta_0)$ only the zero and pole of $\Delta_0$ are ramified and these correspond to cusps ($j=\infty$), any $P \in V_s$ will be adjacent to exactly $[L(Y_0):L(\Delta_0)]=q$ vertices, all in $V_s$. Hence we obtain the following:

\begin{theorem} \label{thm:modularcomponent}
The set $V_s \subset V$ corresponding to the supersingular points of $X_0(T)$ gives rise to a $q$-regular component of $G$ of size $(q^d-1)/(q-1)$.
\end{theorem}

The graph on the $V_s$ is closely related to the $T$-isogeny Brandt graph on the set of supersingular $j$-invariants.

Next we will define universal polynomials $u_i(s)$ with coefficients in $A$ by a simple recursion. We will show that for any polynomial $p(T)\neq T$, the Deuring polynomial $h_{p(T)}(s)$ in $A$-characteristic $p(T)$ is obtained by reducing $u_d(s)$ modulo $p(T)$, where $d=\deg p(T)$.
The idea is to show, using nice identities satisfied by these polynomials $u_i(s)$, that the roots of the reduction of $u_d(s)$ modulo $p(T)$, also give rise to a $q$-regular component of the graph $G$ above. By a theorem from \cite{beelen}, the graph $G$ can have at most one such component and hence we obtain that $h_{p(T)}(s)$ and the reduction of $u_d(s)$ modulo $p(T)$ agree.

\begin{definition}\label{def:polsplit}
For $i\geq -1$ we define the polynomials $u_i(s)\in A[s]$  recursively by:
$$u_{-1}(s)=0  \makebox{, } u_0(s)=1$$
and
\begin{equation}\label{eq:recursion}
u_{i+1}(s)=(s+T^q)^{q^i}u_i(s)-(T^{q^i}-T)s^{q^i}u_{i-1}(s),\quad \text{for } i\ge 0.
\end{equation}
\end{definition}
For any irreducible polynomial $p(T) \neq T$ we denote by $\kappa_{p(T)}$ the residue field of $p(T)$ and by $\overline{\kappa}_{p(T)}$ its algebraic closure.
Reducing $u_{i}(s)$ modulo $p(T)$, gives rise to a polynomial with coefficients in $\kappa_{p(T)}$. We will denote this polynomial by $u_i^{(p(T))}(s)$. It is easy to verify that $u_i^{(p(T))}(s)$ is monic and has degree $(q^i-1)/(q-1)$.

\begin{proposition}\label{prop:simple}
Let $d \ge 1$ be an integer and $p(T) \in \F[T]$ be a prime different from $T$ of degree $d$. All roots of the polynomial $u_d^{(p(T))}(s)$ are simple. Moreover, $0$ is not a root of $u_d^{(p(T))}(s)$.
\end{proposition}
\begin{proof}
It is easily seen by induction that $u_i(0)=(T)^{q(q^i-1)/(q-1)}$ for $i \ge 0$. Therefore $0$ is not a root of $u_d^{(p(T))}(s)$.

We denote by $u_i'(s)$ the derivative of $u_i(s)$ with respect to $s$. By taking the derivative on both sides of the equality sign in Equation (\ref{eq:recursion}), we see that the sequences
$$u':=(u_0'(s), u_1'(s),u_2'(s),\dots) \ \makebox{and} \ u=(u_0(s), u_1(s),u_2(s),\dots)$$ both satisfy recursion \eqref{eq:recursion}. The same holds for any linear combination $\lambda u'+\mu u$.

For convenience, we write $\alpha = T + \langle p(T) \rangle$ in $\kappa_{p(T)}$. Now suppose that $\rho \in \overline{\kappa}_{p(T)}$ is a root of $u_i^{(p(T))}(s)$ of multiplicity greater than one. Then $\rho$ is a common root of $u_d^{(p(T))}(s)$ and $\left(u_{d}^{(p(T))}\right)'(s)$. Hence we can choose $\lambda$ and $\mu$, not both zero, such that
$$\lambda \left(u_{d-1}^{(p(T))}\right)'(\rho)+\mu u_{d-1}^{(p(T))}(\rho)=0 \ \makebox{and}  \ \lambda \left(u_{d}^{(p(T))}\right)'(\rho)+\mu u_{d}^{(p(T))}(\rho)=0.$$
Note that $\alpha^{q^i} \neq 1$ for $0 \le i <d$, since $\alpha$ is a root of the irreducible polynomial $p(T)$ of degree $d$.
Hence, using equation (\ref{eq:recursion}), we see that $\lambda \left(u_{i}^{(p(T))}\right)'(\rho)+\mu u_{i}^{(p(T))}(\rho)=0$ for any $0 \le i \le d$.
In particular $$0=\lambda \left(u_0^{(p(T))}\right)'(\rho)+\mu u_{0}^{(p(T))}(\rho)=\mu \ \makebox{and} \ 0=\lambda \left(u_1^{(p(T))}\right)'(\rho)+\mu u_{1}^{(p(T))}(\rho)=\lambda+\mu(\rho+\alpha^q).$$
Hence $\lambda=\mu=0$, giving a contradiction.
\end{proof}

The following proposition contains a crucial identity.
\begin{proposition}\label{prop:keyresult}
Let $i\ge 0$ be an integer. Then
\begin{multline*}
u_i(-T^q s(s+1)^{q-1})-\left(T(s+1)\right)^{q^i-1}u_i\left(-T\frac{s^q}{(s+1)^{q-1}}\right)
 \\ =-(T^{q^i}-T)\left(T(s+1)\right)^{q^i-1} u_{i-1}\left(-T\frac{s^q}{(s+1)^{q-1}}\right).
\end{multline*}
\end{proposition}
\begin{proof}
We prove the proposition by induction on $i$. The proof only uses elementary computations which are somewhat tedious and therefore omitted. First of all note that for $i=0$ and $i=1$, the proposition is trivial.
Now suppose that $i \ge 1$ and that the proposition is true for $i$ and $i-1$. Then writing $D_0=-T^qs(s+1)^{q-1}$ and $D_1=-Ts^q/(s+1)^{q-1}$ we have by equation~\eqref{eq:recursion}
\begin{align*}
u_{i+1}(D_0) & = (-T^q s(s+1)^{q-1}+T^q)^{q^i} u_i(D_0)-(T^{q^i}-T)(-T^q s(s+1)^{q-1})^{q^i}u_{i-1}(D_0)
\end{align*}
Using the induction hypothesis, we can express this in the form
\begin{align}\label{eq:intermediate}
u_{i+1}(D_0) & = f_i \, u_{i}\left(D_1\right)+f_{i-1}  \, u_{i-1}\left(D_1\right)+f_{i-2}  \, u_{i-2}\left(D_1\right),
\end{align}
for certain expressions $f_i,f_{i-1},f_{i-2} \in A[s].$
Now using equation (\ref{eq:recursion}), with the variable $s$ replaced by $D_1=-Ts^q/(s+1)^{q-1}$, we can express $u_{i-2}(D_1)$ and $u_{i-1}(D_1)$ in terms of $u_{i}(D_1)$ and $u_{i+1}(D_1)$ in equation (\ref{eq:intermediate}). The resulting expression turns out to be exactly what we needed to complete the induction step.
\end{proof}

\begin{theorem}\label{thm:main} For an irreducible polynomial $p(T)\in A$ with $\deg p(T)=d$ and $p(T)\neq T$, the Deuring polynomial  $h_{p(T)}$ can be obtained as
$$h_{p(T)}(s)=u_d^{(p(T))}(s).$$
\end{theorem}
\begin{proof}
Consider the correspondence \eqref{correspondence} and the associated graph $G$. By Theorem~\ref{thm:modularcomponent}, the set $V_S$ of places corresponding to supersingular points on $X_0(T)$ form a $q$-regular finite component of $G$.

Let $W_s$ be the set of places of $L(\Delta_0)$ corresponding to the roots $u_d^{(p(T))}(\Delta_0)$. We claim that set $W_S$ also forms a $q$-regular finite component of $G$. Finiteness is clear. Let $P\in W_s$ and $\xi=\Delta_0(P)$ be the corresponding root of $u_d^{(p(T))}(\Delta_0)$. Let $(P,Q)$ be an edge in $G$ and $\eta=\Delta_0(Q)$. Then there exists an $a$ in $L$ with
\begin{equation} \label{xieta} \xi=-\gamma(T^q)(a+1)^{q-1}a\quad \text{and}\quad \eta=-\gamma(T)\frac{a^q}{(a+1)^{q-1}}.\end{equation}
Reducing modulo $p(T)$ the identity in Proposition~\ref{prop:keyresult} for $i=d$, and using that $p(T) | T^{q^d}-T$, we obtain
$$u_i^{(p(T))}(-\gamma(T^q) s(s+1)^{q-1})-\left(\gamma(T)(s+1)\right)^{q^i-1}u_i^{(p(T))}\left(-\gamma(T)\frac{s^q}{(s+1)^{q-1}}\right)=0.$$
Substituting $s=a$ and using \eqref{xieta} we get
\begin{multline*}
u_i^{(p(T))}(\xi)=u_i^{(p(T))}(-\gamma(T^q) a(a+1)^{q-1})\\=\left(\gamma(T)(a+1)\right)^{q^i-1}u_i^{(p(T))}\left(-\gamma(T)\frac{a^q}{(a+1)^{q-1}}\right) =\left(\gamma(T)(a+1)\right)^{q^i-1}u_i^{(p(T))}(\eta).
\end{multline*}
Since $a\neq -1$ (otherwise $\xi=0$, contradicting Proposition~\ref{prop:simple}), we have $u_i^{(p(T))}(\eta)=0$, hence $Q$ is in $W_s$. Similarly, $(P,Q)$ is an edge in $G$ and $Q\in W_s$ implies $P\in W_s$.
To see the $q$-regularity of $W_s$ just note that the places in $W_s$ are not ramified in the extensions of the correspondence.

Finally, $V_s=W_s$, by \cite[Corollary 4.7]{beelen}, since the graph $G$ can have at most one $q$-regular finite component.
Both $h_{p(T)}(s)$ and $u_d^{(p(T))}(s)$ are monic, hence the result follows.
\end{proof}
\begin{remark} Note that the proof of Theorem~\ref{thm:main} also shows that the subgraph on supersingular points $V_s$ is connected.
\end{remark}

Using \eqref{deuringrelation}, the Deuring Polynomial $H_{p(T)}(s)$ will be given by
$$H_{p(T)}(s):=\Bigl(\dfrac{(s^q-s)^{q-1}}{\gamma(T^q)}\Bigr)^{\frac{q^d-1}{q-1}}u_d^{(p(T))}\Bigl(\frac{\gamma(T)}{(s^q-s)^{q-1}}\Bigr).$$
Alternatively, by a change of variables, we can obtain recurrence relations for $H_{p(T)}(s)$ directly :

\begin{corollary}
Define $U_{-1}(s)=0  \makebox{, } U_0(s)=1$ and for $i\geq 0$
$$U_{i+1}(s)=\bigl((s^q-s)^{q-1}+\frac{1}{T^{q-1}}\bigr)^{q^i}U_i(s)-\frac{T^{q^i}-T}{T^{q^{i+1}}}(s^q-s)^{(q-1)q^{i-1}}U_{i-1}(s).$$
Let $p(T)\neq T$ be an irreducible polynomial in $A$ with $\deg p(T)=d$. Then
$$H_{p(T)}(s)=U_d^{(p(T))}(s),$$
where $U_d^{(p(T))}(s)$ denotes the reduction of $U_d(s)$ modulo $p(T)$.
\end{corollary}

\bibliographystyle{99}

\vspace{3ex}

\noindent
Alp Bassa\\

Bo\u{g}azi\c{c}i University,

Faculty of Arts and Sciences,

Department of Mathematics,

34342 Bebek, \.{I}stanbul,

Turkey

alp.bassa@boun.edu.tr

\vspace{3ex}

\noindent
Peter Beelen\\

Technical University of Denmark,

Department of Applied Mathematics and Computer Science,

Matematiktorvet 303B, 2800 Kgs. Lyngby,

Denmark

pabe@dtu.dk

\end{document}